\documentclass[oneside,final,11pt]{amsart}

\usepackage{dsfont}

\newcommand{\Z}{\mathds Z}

\DeclareMathOperator{\Ann}{ann}

\title[A non divisible LCA with torsion free divisible character group]{A locally compact non divisible
abelian group whose character group is torsion free and divisible}

\author{Daniel V. Tausk}
\address{Departamento de Matem\'atica,\hfill\break\indent Universidade de S\~ao Paulo, Brazil}
\email{tausk@ime.usp.br} \urladdr{http://www.ime.usp.br/\~{}tausk}

\date{January 9th, 2011}

\begin{document}

\theoremstyle{plain}\newtheorem{teo}{Theorem}
\theoremstyle{plain}\newtheorem{prop}[teo]{Proposition}
\theoremstyle{plain}\newtheorem{lem}[teo]{Lemma}
\theoremstyle{plain}\newtheorem{cor}[teo]{Corollary}
\theoremstyle{definition}\newtheorem{defin}[teo]{Definition}
\theoremstyle{remark}\newtheorem{rem}[teo]{Remark}
\theoremstyle{plain} \newtheorem{assum}[teo]{Assumption}
\swapnumbers
\theoremstyle{definition}\newtheorem{example}{Example}[section]

\begin{abstract}
It has been claimed by Halmos in \cite{Halmos} that if $G$ is a Hausdorff locally compact topological abelian
group and if the character group of $G$ is torsion free then $G$ is divisible. We prove that such claim is false, by
presenting a family of counterexamples. While other counterexamples are known (see \cite[4.16]{Armacost}), we also present a family of stronger counterexamples,
showing that even if one assumes that the character group of $G$ is both torsion free and divisible, it does not follow that $G$ is divisible.
\end{abstract}

\maketitle

\begin{section}{Introduction}\label{sec:intro}

Let $G$ be an abelian group\footnote{%
Except for the circle group $S^1$, abelian groups will be written additively.}.
Given an integer $n$, we denote by $nG$ and by $G[n]$ the subgroups of $G$ defined by:
\[nG=\big\{nx:x\in G\big\},\quad G[n]=\big\{x\in G:nx=0\big\}.\]
If $G$ is an abelian topological group, then its {\em character group\/} $\hat G$ is the abelian group of all continuous homomorphisms
$\xi:G\to S^1$, where $S^1$ is the (multiplicative) circle group of unitary complex numbers; the group $\hat G$ is endowed with the compact-open
topology. The celebrated {\em Pontryagin duality theorem\/} (see, for instance, \cite{Morris}) states that if $G$ is a Hausdorff locally compact abelian topological
group then its character group $\hat G$ is a Hausdorff locally compact abelian topological group as well and the character group of $\hat G$ is $G$ itself;
more precisely, the map that associates to each $x\in G$ the evaluation map $\hat G\ni\xi\mapsto\xi(x)\in S^1$ is a homeomorphic isomorphism between $G$ and
the character group of $\hat G$.

If $H$ is a subgroup of $G$ then the {\em annihilator\/} of $H$ is the subgroup $\Ann(H)$ of $\hat G$ consisting of all characters
$\xi:G\to S^1$ that are trivial over $H$. Clearly, given an integer $n$, then:
\[\Ann(nG)=\hat G[n].\]
In particular, if $G$ is divisible (i.e., if $nG=G$ for every non zero integer $n$) then its character group
$\hat G$ is torsion free (i.e., $\hat G[n]$ is trivial for every non zero integer $n$). It has been claimed
by Halmos in \cite{Halmos} that the converse is true if $G$ is Hausdorff locally compact. The argument presented
in \cite{Halmos} has a gap: if $\hat G$ is torsion free then $\Ann(nG)$ is trivial for every non zero integer $n$,
but that in principle implies only that $nG$ is dense\footnote{%
If $\Ann(nG)$ is trivial then $nG$ is indeed dense in $G$. Otherwise, Pontryagin duality would give us a non trivial character on the (non trivial) quotient
of $G$ by the closure of $nG$ and such non trivial character would correspond to a non trivial element of $\Ann(nG)$.}
in $G$, not that $nG=G$. It should be observed, however, that the claim made by Halmos
is true if $G$ is either compact or discrete and that the proof of the main result of \cite{Halmos} is not affected by the incorrect claim.

In Section~\ref{sec:firstce}, we will present a family of examples of Hausdorff locally
compact abelian topological groups $G$ such that $nG$ is dense
in $G$ for every non zero integer $n$, but such that $nG\ne G$ for
some non zero integer $n$. In particular, any such group $G$ is an example of a Hausdorff locally compact abelian
topological group that is not divisible, but whose character group is torsion free. While other examples of that phenomenon are known (see \cite[4.16]{Armacost}),
in Section~\ref{sec:stronger} we will also present a family of examples of Hausdorff locally compact abelian topological groups $G$ which are {\em both\/}
divisible and torsion free, but such that $\hat G$ is (torsion free but) not divisible. In particular, by Pontryagin duality, it follows that $\hat G$
is a Hausdorff locally compact abelian topological group whose
character group (which is isomorphic to $G$) is {\em both\/} divisible and torsion free, but still $\hat G$ is not divisible.

\end{section}

\begin{section}{Extending the topology of a subgroup}\label{sec:exttopology}

Let us start by presenting a general construction of a topology on an abelian group from a topology on a given subgroup (the construction is well-known,
see for instance \cite{ClDooSch, ClSch1, ClSch2}).
Let $G$ be an abelian group and $H$ be a subgroup of $G$. Assume that $H$ is endowed with a topology that makes it
into a topological group. We claim that there exists a unique topology on $G$ such that:
\begin{itemize}
\item[(a)] $G$ is a topological group;
\item[(b)] the given topology of $H$ is inherited from $G$;
\item[(c)] $H$ is open in $G$.
\end{itemize}
Such a topology is constructed as follows.
Given $g\in G$, then the coset $g+H$ of $H$ can be endowed with a topology by requiring that the translation map:
\[L_g:H\ni x\longmapsto g+x\in g+H\]
be a homeomorphism. The fact that the translation maps of $H$ are homeomorphisms of $H$ implies that the topology
defined on the coset $g+H$ does not depend on the representative $g$ of the coset. We topologize $G$ by making it
the topological sum of the cosets $g+H$, $g\in G$; that is, we say that $U$ is open in $G$ if $U\cap(g+H)$ is open in
$g+H$ for every $g\in G$. One readily checks that such a topology is the only topology on $G$ satisfying (a), (b) and (c).
Notice that, since the cosets of $H$ are all homeomorphic to $H$ and open in $G$, it follows that if $H$ is
Hausdorff then so is $G$. Moreover, since every compact neighborhood of the neutral element in $H$ is also a
compact neighborhood of the neutral element in $G$, it follows that $G$ is locally compact if $H$ is locally compact.

\end{section}

\begin{section}{The first family of counterexamples}\label{sec:firstce}

Let $A$ be a Hausdorff compact abelian topological group that is not divisible and let $B$ be a divisible abelian
group such that $A$ is a subgroup of $B$ (for instance, let $B=S^1$ and $A$ be a non trivial finite subgroup of $S^1$
endowed with the discrete topology). Let $B^\omega$ denote the group of all sequences $(x_k)_{k\in\omega}$ of elements of $B$
and let $G$ denote the subgroup of $B^\omega$ consisting of those sequences $(x_k)_{k\in\omega}$ such that $x_k$
is in $A$ for $k$ sufficiently large. Let $H=A^\omega$ denote the subgroup of $G$ consisting of sequences in $A$.
We endow $H$ with the product topology and $G$ with the unique topology satisfying (a), (b) and (c) of Section~\ref{sec:exttopology}.
Then $H$ is a Hausdorff compact topological group and thus $G$ is a Hausdorff locally compact topological group.
If $n$ is a non zero integer then the subgroup $nG$ of $G$ consists of those sequences $(x_k)_{k\in\omega}$
such that $x_k$ is in $nA$ for $k$ sufficiently large. If $n_0$ is a non zero integer such that $n_0A\ne A$ then $n_0G\ne G$
and therefore $G$ is not divisible.
We will show that if $n$ is a non zero integer then $nG$ is dense in $G$ and from this it will follow from the discussion
at the introduction that the character group $\hat G$ is torsion free. Let $J$ denote the subgroup of $G$ consisting of
sequences $(x_k)_{k\in\omega}$ in $B$ that are trivial for $k$ sufficiently large. Since $J$ is obviously contained in $nG$ for any non zero integer $n$, it suffices
to prove that $J$ is dense in $G$ in order to establish that $nG$ is dense in $G$ for every non zero integer $n$.
Clearly, $G=H+J$, so that $J$ intersects every coset of $H$. Now let us prove that $J$ is dense in $G$ by proving
that $J\cap(x+H)$ is dense in $x+H$, for every coset $x+H$ of $H$ in $G$.
Since the coset $x+H$ intersects $J$, we can assume that $x\in J$. Thus, the translation map $L_x:H\to x+H$ is
a homeomorphism that carries $J\cap H$ to $J\cap(x+H)$. From the definition of the product topology, it is obvious
that $J\cap H$ is dense in $H$ and therefore $J\cap(x+H)$ is dense in $x+H$. This concludes the proof that
the subgroup $J$ is dense in $G$.

\end{section}

\begin{section}{The family of stronger counterexamples}\label{sec:stronger}

We will now present an example of a Hausdorff locally compact abelian topological group $G$ that is both
divisible and torsion free, but such that its character group $\hat G$ is not divisible. We need a couple
of preliminary lemmas.

\begin{lem}\label{lema:1}
Let $G$ be an abelian divisible topological group. If there exists an open subgroup $H$ of $G$, a non zero
integer $n$ and a {\em discontinuous\/} homomorphism $\phi:H\to S^1$ that is trivial over $nH$ then the character
group $\hat G$ is not divisible.
\end{lem}
\begin{proof}
Since $S^1$ is divisible, $\phi$ extends to a (obviously discontinuous) homomorphism $\phi':G\to S^1$. Consider the
homomorphism $\xi:G\to S^1$ defined by $\xi(x)=\phi'(nx)$, for all $x\in G$. Then $\xi$ is trivial over $H$ and,
since $H$ is open, $\xi$ is continuous. Assuming by contradiction that $\hat G$ is divisible, we can find
a continuous homomorphism $\alpha:G\to S^1$ such that $\alpha(x)^n=\alpha(nx)=\xi(x)$, for all $x\in G$. Then
$\alpha$ and $\phi'$ are equal over $nG$ and since $G$ is divisible, we obtain that $\alpha=\phi'$, contradicting the
continuity of $\alpha$.
\end{proof}

\begin{lem}\label{lema:2}
Let $K$ be an abelian group endowed with a topology\footnote{%
It is not relevant that $K$ be a topological group, i.e., the continuity of the operations of $K$ is not used in the proof.}.
If $K$ admits a proper dense subgroup $D$ then there exists a discontinuous homomorphism from $K$ to $S^1$.
\end{lem}
\begin{proof}
Since $K/D$ is a non trivial abelian group, there exists a non trivial homomorphism $\phi:K/D\to S^1$ (start
with a non trivial $S^1$-valued homomorphism defined over a non trivial cyclic subgroup of $K/D$ and then extend
it to all of $K/D$ using the fact that $S^1$ is divisible). The composition of $\phi$ with the quotient map $K\to K/D$ is
a non trivial homomorphism that is trivial over $D$, and therefore it must be discontinuous.
\end{proof}

\begin{cor}\label{thm:cor}
Let $G$ be an abelian divisible topological group. If there exists an open subgroup $H$ of $G$ and a non zero
integer $n$ such that $H/nH$ (endowed with the quotient topology) has a proper dense subgroup then the character group
$\hat G$ is not divisible.
\end{cor}
\begin{proof}
By Lemma~\ref{lema:2}, there exists a discontinuous $S^1$-valued homomorphism over $H/nH$; its composition with
the quotient map $H\to H/nH$ is a discontinuous $S^1$-valued homomorphism over $H$ that is trivial over $nH$.
The conclusion follows from Lemma~\ref{lema:1}.
\end{proof}

The construction of our family of stronger counterexamples goes as follows. Let $A$ be a Hausdorff compact abelian
non divisible topological group and let $B$ be a torsion free divisible abelian group such that $A$ is a subgroup of $B$.
A concrete example of groups $A$, $B$ satisfying the required conditions will be supplied at the end of the section.
Let $H=A^\omega$ denote the group of all sequences in $A$ endowed with the product topology and
let $G=B^\omega$ be the group of all sequences in $B$, endowed with the unique topology satisfying
(a), (b) and (c) of Section~\ref{sec:exttopology}. The group $H$ is Hausdorff compact
and thus $G$ is Hausdorff locally compact; moreover, like $B$, the group $G$ is both divisible and torsion free.
We use Corollary~\ref{thm:cor} to establish that the character group $\hat G$ is not divisible. Let $n$
be a non zero integer such that $nA\ne A$. We claim that if $H/nH$ is endowed with the quotient topology then
it has a proper dense subgroup. First, we check that the quotient topology of $H/nH$ coincides with the product topology
of $(A/nA)^\omega$, each factor $A/nA$ being endowed with the quotient topology. Namely, if $A/nA$ is endowed with the
quotient topology, then the quotient map $A\to A/nA$ is continuous, open and surjective; therefore, if
$H/nH\cong(A/nA)^\omega$ is endowed with the product topology, then the quotient map $H\to H/nH$ is also continuous,
open and surjective and therefore it is a topological quotient map. This observation proves that the product topology
of $(A/nA)^\omega$ coincides with the quotient topology of $H/nH$. Now, it follows directly from the definition of
the product topology that the subgroup of $H/nH\cong(A/nA)^\omega$ consisting of sequences $(x_k)_{k\in\omega}$ that
are trivial for $k$ sufficiently large is a (proper) dense subgroup. This concludes the proof that $\hat G$ is
not divisible.

Finally, let us present a concrete example of groups $A$, $B$ satisfying the required conditions.
Let $A$ be the group of $p$-adic integers (where $p$ is some fixed prime number) and
$B$ be the $p$-adic field. We have $pA\ne A$, so that $A$ is not divisible; moreover, $B$ is a field of characteristic
zero, so that it is both torsion free and divisible as an abelian group. The fact that $A$ can be made into
a Hausdorff compact topological group follows from the observation that $A$ is (isomorphic to) the character group of the discrete
$p$-quasicyclic group $\Z(p^\infty)$ of elements of $S^1$ whose order is a power of $p$
(see, for instance, \cite[Proposition~3.1]{Harrison}) and that the character group of a discrete topological group is compact.

\end{section}

\end{document}